\documentclass{article}
\usepackage{amsfonts}
\usepackage{amsmath}
\usepackage{amssymb}
\usepackage{graphicx}%
\providecommand{\U}[1]{\protect\rule{.1in}{.1in}}
\newtheorem{theorem}{Theorem}

\newtheorem{corollary}[theorem]{Corollary}

\newtheorem{lemma}[theorem]{Lemma}

\newtheorem{proposition}[theorem]{Proposition}
\newtheorem{remark}[theorem]{Remark}

\newenvironment{proof}[1][Proof]{\noindent\textbf{#1.} }{\ \rule{0.5em}{0.5em}}
\begin{document}

\title{\textit{On a singular minimizing problem}}
\author{{\small \textbf{Grey Ercole\thanks{Corresponding author}\ \ and Gilberto de
Assis Pereira}}\\{\small \textit{Departamento de Matem\'{a}tica - Universidade Federal de Minas
Gerais}}\\{\small \textit{Belo Horizonte, MG, 30.123-970, Brazil.}}\\{\small grey@mat.ufmg.br, gilbertoapereira@yahoo.com.br}}
\maketitle

\begin{abstract}
For each $q\in(0,1)$ let
\[
\lambda_{q}(\Omega):=\inf\left\{  \left\Vert \nabla v\right\Vert
_{L^{p}(\Omega)}^{p}:v\in W_{0}^{1,p}(\Omega)\;\mathrm{and}\;\int_{\Omega
}\left\vert v\right\vert ^{q}\mathrm{d}x=1\right\}  ,
\]
where $p>1$ and $\Omega$ is a bounded and smooth domain of $\mathbb{R}^{N},$
$N\geq2.$ We first show that%
\[
0<\mu(\Omega):=\lim_{q\rightarrow0^{+}}\lambda_{q}(\Omega)\left\vert
\Omega\right\vert ^{\frac{p}{q}}<\infty,
\]
where $\left\vert \Omega\right\vert =\int_{\Omega}\mathrm{d}x.$ Then, we prove
that
\[
\mu(\Omega)=\min\left\{  \left\Vert \nabla v\right\Vert _{L^{p}(\Omega)}%
^{p}:v\in W_{0}^{1,p}(\Omega)\;\mathrm{and}\;\lim_{q\rightarrow0^{+}}\left(
\frac{1}{\left\vert \Omega\right\vert }\int_{\Omega}\left\vert v\right\vert
^{q}\mathrm{d}x\right)  ^{\frac{1}{q}}=1\right\}
\]
and that $\mu(\Omega)$ is reached by a function $u\in W_{0}^{1,p}(\Omega),$
which is positive in $\Omega,$ belongs to $C^{0,\alpha}(\overline{\Omega}),$
for some $\alpha\in(0,1),$ and satisfies%
\[
\mathrm{\ }-\operatorname{div}(\left\vert \nabla u\right\vert ^{p-2}\nabla
u)=\mu(\Omega)\left\vert \Omega\right\vert ^{-1}u^{-1}\quad\mathrm{in\quad
\Omega,\quad and\quad}\int_{\Omega}\log u\mathrm{d}x=0.
\]
We also show that $\mu(\Omega)^{-1}$ is the best constant $C$ in the following
log-Sobolev type inequality%
\[
\exp\left(  \frac{1}{\left\vert \Omega\right\vert }\int_{\Omega}\log\left\vert
v\right\vert ^{p}\mathrm{d}x\right)  \leq C\left\Vert \nabla v\right\Vert
_{L^{p}(\Omega)}^{p},\quad v\in W_{0}^{1,p}(\Omega)
\]
and that this inequality becomes an equality if, and only if, $v$ is a scalar
multiple of $u$ and $C=\mu(\Omega)^{-1}.\medskip$



\end{abstract}

\section{Introduction}

Let $p>1$ be fixed and let $\Omega\subset\mathbb{R}^{N},$ $N\geq2,$ be a
bounded and smooth domain. For each $q\in(0,1)$ let us define
\begin{equation}
\lambda_{q}(\Omega):=\inf\left\{  \left\Vert \nabla v\right\Vert _{p}^{p}:v\in
W_{0}^{1,p}(\Omega)\;\mathrm{and}\;\int_{\Omega}\left\vert v\right\vert
^{q}\mathrm{d}x=1\right\}  , \label{lambq}%
\end{equation}
where $\left\Vert \cdot\right\Vert _{s}$ denotes the standard norm of the
Lebesgue space $L^{s}(\Omega),$ $1\leq s\leq\infty.$

As proved in \cite{Anello}, $\lambda_{q}(\Omega)$ is achieved by a positive
function $u_{q}\in W_{0}^{1,p}(\Omega)\cap C^{1}(\Omega)$ satisfying the
singular Dirichlet problem
\begin{equation}
\left\{
\begin{array}
[c]{ll}%
-\Delta_{p}v=\lambda_{q}(\Omega)\left\vert v\right\vert ^{q-2}v &
\mathrm{in\ }\Omega\\
v=0 & \mathrm{on\ }\partial\Omega,
\end{array}
\right.  \label{qbvp}%
\end{equation}
in the weak sense, where $\Delta_{p}v=\operatorname{div}\left(  \left\vert
\nabla v\right\vert ^{p-2}\nabla v\right)  $ is the $p$-Laplacian operator.
Moreover, it follows from \cite[Theorem 1.1 (i)]{GST1} that $u_{q}\in
C^{1,\alpha}(\overline{\Omega}),$ for some $\alpha\in(0,1).$

In this paper we first show that
\begin{equation}
0<\mu(\Omega):=\lim_{q\rightarrow0^{+}}\lambda_{q}(\Omega)\left\vert
\Omega\right\vert ^{\frac{p}{q}}<\infty, \label{mu1}%
\end{equation}
where $\left\vert D\right\vert $ stands for the $N$-dimensional Lebesgue
volume of $D\subset\mathbb{R}^{N},$ \textrm{i.\thinspace e.} $\left\vert
D\right\vert =\int_{D}\mathrm{d}x.$

Then, we prove that
\begin{equation}
\mu(\Omega)=\min\left\{  \left\Vert \nabla v\right\Vert _{p}^{p}:v\in
W_{0}^{1,p}(\Omega)\;\mathrm{and}\;\lim_{q\rightarrow0^{+}}\left(  \frac
{1}{\left\vert \Omega\right\vert }\int_{\Omega}\left\vert v\right\vert
^{q}\mathrm{d}x\right)  ^{\frac{1}{q}}=1\right\}  \label{defmu}%
\end{equation}
and that the minimum is reached by a function $u\in W_{0}^{1,p}(\Omega),$
which is positive in $\Omega,$ belongs to $C^{0,\alpha}(\overline{\Omega}),$
for some $\alpha\in(0,1),$ and satisfies:

\begin{enumerate}
\item[(i)] $u=\lim_{q\rightarrow0^{+}}\left\vert \Omega\right\vert ^{\frac
{1}{q}}u_{q}$ in $W_{0}^{1,p}(\Omega);$

\item[(ii)] $-\Delta_{p}u=\mu(\Omega)\left\vert \Omega\right\vert ^{-1}u^{-1}$
in $\Omega;$ and

\item[(iii)] $\int_{\Omega}\log u\mathrm{d}x=0.$
\end{enumerate}

Exploring (\ref{defmu}) we also prove that $\mu(\Omega)^{-1}$ is the best
constant $C$ in the following log-Sobolev type inequality%
\[
\exp\left(  \frac{1}{\left\vert \Omega\right\vert }\int_{\Omega}\log\left\vert
v\right\vert ^{p}\mathrm{d}x\right)  \leq C\left\Vert \nabla v\right\Vert
_{p}^{p},\quad v\in W_{0}^{1,p}(\Omega),
\]
and that $\mu(\Omega)^{-1}$ is reached if, and only if, $v$ is a scalar
multiple of $u$, which is the unique case where the inequality becomes an
equality. Up to our knowledge, these facts are entirely new.

It is easy to check that for each fixed $\lambda>0$ the function $u_{\lambda
}:=\left(  \frac{\lambda\left\vert \Omega\right\vert }{\mu(\Omega)}\right)
^{\frac{1}{p}}u\ $is a positive weak solution of the singular problem
\begin{equation}
\left\{
\begin{array}
[c]{ll}%
-\Delta_{p}v=\lambda v^{-1} & \mathrm{in\ }\Omega,\\
v=0 & \mathrm{on\ }\partial\Omega.
\end{array}
\right.  \label{lsing}%
\end{equation}

The function $u_{\lambda}$ is, in fact, the unique positive solution of
(\ref{lsing}). This uniqueness result follows from a simple and well-known
inequality involving vectors of $\mathbb{R}^{N}$. Existence and regularity of
weak solutions for (\ref{lsing}) were first studied in the particular case
$p=2$ (see \cite{CRT, LM, Stuart}), whereas the case $p>1$ has received more
attention in the last decade (see \cite{Chu, GST, GST1, Mohammed} and
references therein).

We remark that the differentiability of the functional $v\in W_{0}%
^{1,p}(\Omega)\mapsto\lambda%
{\displaystyle\int_{\Omega}}
\log\left\vert v\right\vert \mathrm{d}x\in\lbrack-\infty,\infty)$ is a
delicate question, which makes it difficult to apply variational methods to
obtain the positive solution of (\ref{lsing}). Thus, $u_{\lambda}$ has
generally been obtained by nonvariational methods, mainly the sub-super
solution method. As for regularity, it is proved in \cite[Theorem 2.2
(ii)]{GST1} that $u_{\lambda}\in C^{0,\alpha}(\overline{\Omega}),$ for some
$\alpha\in(0,1).$

We emphasize that besides providing a new existence proof of $u_{\lambda},$ we
show that
\[
\int_{\Omega}\log u_{\lambda}\mathrm{d}x=\frac{\left\vert \Omega\right\vert
}{p}\log\left(  \frac{\lambda\left\vert \Omega\right\vert }{\mu(\Omega
)}\right)  \in(-\infty,\infty).
\]
This property of $u_{\lambda}$ was not known up to now. It comes from the
connection between (\ref{lsing}) and the minimizing problem (\ref{defmu}).

Also in this paper, we show that the formal energy functional associated with
(\ref{lsing}),
\[
J_{\lambda}(v):=\left\{
\begin{array}
[c]{ll}%
\dfrac{1}{p}%
{\displaystyle\int_{\Omega}}
\left\vert \nabla v\right\vert ^{p}\mathrm{d}x-\lambda%
{\displaystyle\int_{\Omega}}
\log\left\vert v\right\vert \mathrm{d}x, & \mathrm{if}\;%
{\displaystyle\int_{\Omega}}
\log\left\vert v\right\vert \mathrm{d}x\in(-\infty,\infty)\\
\infty, & \mathrm{if}\;%
{\displaystyle\int_{\Omega}}
\log\left\vert v\right\vert \mathrm{d}x=-\infty,
\end{array}
\right.
\]
attains its minimum value $\dfrac{\lambda\left\vert \Omega\right\vert }%
{p}\left(  1-\log\left(  \frac{\lambda\left\vert \Omega\right\vert }%
{\mu(\Omega)}\right)  \right)  $ only at the functions $u_{\lambda}$ and
$-u_{\lambda}.$

We end the paper by describing the asymptotic behavior of the pair
$(\lambda_{q}(\Omega),\left\Vert u_{q}\right\Vert _{\infty}),$ as
$q\rightarrow0^{+}.$ That is, we determine when these quantities either go to
$0$ or to $\infty$ or remain bounded from above and from below, when
$q\rightarrow0^{+}.$ More precisely, we obtain directly from (\ref{mu1}) that
\begin{equation}
\lim_{q\rightarrow0^{+}}\lambda_{q}(\Omega)=\left\{
\begin{array}
[c]{ll}%
\infty & \mathrm{if}\;\left\vert \Omega\right\vert <1\\
\mu(\Omega) & \mathrm{if}\;\left\vert \Omega\right\vert =1\\
0 & \mathrm{if}\;\left\vert \Omega\right\vert >1,
\end{array}
\right.  \label{assy1}%
\end{equation}
and apply lower and upper estimates (derived in Section \ref{S1}) to show that%
\begin{equation}
\lim_{q\rightarrow0^{+}}\left\Vert u_{q}\right\Vert _{\infty}=\left\{
\begin{array}
[c]{ll}%
\infty & \mathrm{if}\quad\left\vert \Omega\right\vert <1\\
0 & \mathrm{if}\quad\left\vert \Omega\right\vert >1
\end{array}
\right.  \label{assy2}%
\end{equation}
and that
\begin{equation}
0<A\mu(\Omega)^{\frac{1}{p}}\leq\lim_{q\rightarrow0^{+}}\left\Vert
u_{q}\right\Vert _{\infty}\leq B\mu(\Omega)^{\frac{1}{p}},\quad\mathrm{if}%
\;\left\vert \Omega\right\vert =1, \label{assy3}%
\end{equation}
where $A$ and $B$ are positive constants that depend only on $N$ and $p.$

The result in (\ref{assy1}) for the case $\left\vert \Omega\right\vert <1$ has
recently been obtained in \cite{Anello}. The cases $\left\vert \Omega
\right\vert \geq1$ in (\ref{assy1}) as well as (\ref{assy2}) and (\ref{assy3})
are new observations.

Thus, (\ref{assy1}), (\ref{assy2}) and (\ref{assy3}) provide complementary
information on how the function $q\in(0,p^{\star})\mapsto(\lambda_{q}%
(\Omega),\left\Vert u_{q}\right\Vert _{\infty})\in\mathbb{R}^{2}$ behaves at
the endpoints of its domain. In fact, the behavior of this function as
$q\rightarrow p^{\star}$ is well known:
\begin{equation}
\lim_{q\rightarrow p^{\star}}\lambda_{q}(\Omega)=\left\{
\begin{array}
[c]{ll}%
S_{N,p}, & \;\mathrm{if}\quad1<p<N\\
0, & \;\mathrm{if}\quad p=N>1\\
\Lambda_{p}(\Omega), & \;\mathrm{if}\quad p>N,
\end{array}
\right.  \label{asyl}%
\end{equation}
and
\begin{equation}
\lim_{q\rightarrow p^{\star}}\left\Vert u_{q}\right\Vert _{\infty}=\left\{
\begin{array}
[c]{ll}%
\infty, & \;\mathrm{if}\quad1<p<N\\
C_{N}, & \;\mathrm{if}\quad p=N>1\\
1, & \;\mathrm{if}\quad p>N,
\end{array}
\right.  \label{asyu}%
\end{equation}
where
\[
\Lambda_{p}(\Omega):=\min\left\{  \left\Vert \nabla u\right\Vert _{p}^{p}:u\in
W_{0}^{1,p}(\Omega)\text{ \ and \ }\left\Vert u\right\Vert _{\infty
}=1\right\}  ,
\]
$S_{N,p}$ is the well-known Sobolev constant, defined by%
\begin{equation}
S_{N,p}:=\pi^{\frac{p}{2}}N\left(  \frac{N-p}{p-1}\right)  ^{p-1}\left(
\frac{\Gamma(N/p)\Gamma(1+N-N/p)}{\Gamma(1+N/2)\Gamma(N)}\right)  ^{\frac
{p}{N}}, \label{SNp}%
\end{equation}
$\Gamma$ denoting the Gamma Function, and $C_{N}$ is a positive constant that
does not depend on $\Omega.$

For (\ref{asyl}) and (\ref{asyu}) we refer to \cite{Anello,GEJMAA},
\cite{RenWei-p} and \cite{GGilb}, respectively to the cases $1<p<N,$ $p=N>1$
and $p>N>1.$

\section{Preliminaries\label{S1}}

In this section, we present some properties of the weak solutions of the
singular Dirichlet problem
\begin{equation}
\left\{
\begin{array}
[c]{ll}%
-\Delta_{p}v=\lambda v^{q-1} & \mathrm{in\ }\Omega,\quad0\leq q<1,\quad
\lambda>0\\
v>0 & \mathrm{in\ }\Omega,\\
v=0 & \mathrm{on\ }\partial\Omega
\end{array}
\right.  \label{singdiric}%
\end{equation}
which will be used in the paper. A weak solution of (\ref{singdiric}) is a
function $v\in W_{0}^{1,p}(\Omega)$ such that $\operatorname{essinf}_{K}v>0$
in each compact $K\subset\Omega$ and
\begin{equation}
\int_{\Omega}\left\vert \nabla v\right\vert ^{p-2}\nabla v\cdot\nabla
\varphi\mathrm{d}x=\lambda\int_{\Omega}v^{q-1}\varphi\mathrm{d}%
x,\;\mathrm{for\ all}\;\varphi\in W_{0}^{1,p}(\Omega).\label{weak+}%
\end{equation}

Next, we present a simple uniqueness proof for (\ref{singdiric}), which makes
use of the following well-known inequality:
\begin{equation}
(\left\vert x\right\vert ^{p-2}x-\left\vert y\right\vert ^{p-2}y)\cdot
(x-y)\geq0\;\mathrm{for\ all}\;x,y\in\mathbb{R}^{N}. \label{vectors}%
\end{equation}

\begin{proposition}
\label{propuniq}Let $u_{1},u_{2}\in W_{0}^{1,p}(\Omega)$ be weak solutions of
(\ref{singdiric}). Then, $u_{2}=u_{1}$ \textrm{a.\thinspace e.} in $\Omega.$
\end{proposition}

\begin{proof}
Taking $\varphi=u_{2}-u_{1}$ in (\ref{weak+}) we obtain
\begin{equation}
\int_{\Omega}(\left\vert \nabla u_{2}\right\vert ^{p-2}\nabla u_{2}-\left\vert
\nabla u_{1}\right\vert ^{p-2}\nabla u_{1})\cdot\nabla(u_{2}-u_{1}%
)\mathrm{d}x=\lambda\int_{\Omega}\left(  u_{2}^{q-1}-u_{1}^{q-1}\right)
(u_{2}-u_{1})\mathrm{d}x. \label{weak++}%
\end{equation}

It follows from (\ref{vectors}) that the integrand in the left-hand side of
(\ref{weak++}) is nonnegative. It is easy to see that the integrand of the
right-hand side of (\ref{weak++}) cannot be positive. Thus, both of them must
be null almost everywhere in $\Omega,$ which implies that $u_{2}=u_{1}$
\textrm{a.\thinspace e.} in $\Omega.$
\end{proof}

In the sequel we derive estimates for the weak solutions of (\ref{singdiric})
depending explicitly on $q\in\lbrack0,1).$

Let us recall that
\begin{equation}
\lambda_{q}(D^{\ast})\leq\lambda_{q}(D),\quad0<q<p^{\star} \label{aux1}%
\end{equation}
where $D$ is a general bounded and smooth domain of $\mathbb{R}^{N}$ and
$D^{\ast}$ is the ball centered at the origin and with the same volume as $D,$
that is, $\left\vert D^{\ast}\right\vert =\left\vert D\right\vert .$

Inequality (\ref{aux1}) comes from well known properties of Schwarz
symmetrization (see \cite{Kw}) and, among other important utilities, it
provides a lower bound for $\lambda_{q}(D)$ in terms of $\left\vert
D\right\vert $ and $\lambda_{q}(B_{1}),$ where $B_{1}$ denotes the unit ball
of $\mathbb{R}^{N}.$ In fact, one can show that
\begin{equation}
\lambda_{q}(D^{\ast})=\lambda_{q}(B_{1})\left(  \frac{\left\vert D^{\ast
}\right\vert }{\omega_{N}}\right)  ^{1-\frac{p}{N}-\frac{p}{q}},\quad
0<q<p^{\star} \label{aux2}%
\end{equation}
where $\omega_{N}=\left\vert B_{1}\right\vert .$

Hence, by combining (\ref{aux1}) and (\ref{aux2}) one obtains the following
version of the well-known Poincar\'{e}-Sobolev inequality%
\[
\left(  \int_{D}\left\vert v\right\vert ^{q}\mathrm{d}x\right)  ^{\frac{p}{q}%
}\leq\frac{\left\vert D\right\vert ^{\frac{p}{q}+\frac{p}{N}-1}}{\lambda
_{q}(B_{1})\left\vert B_{1}\right\vert ^{\frac{p}{q}+\frac{p}{N}-1}}\left\Vert
\nabla v\right\Vert _{L^{p}(D)}^{p},\;\mathrm{for\ all}\;v\in W_{0}%
^{1,p}(D)\;\mathrm{and}\;0<q<p^{\star}.
\]
When $q=p$ we have
\begin{equation}
\int_{D}|v|^{p}\mathrm{d}x\leq\frac{|D|^{\frac{p}{N}}}{C_{N,p}}\int_{D}|\nabla
v|^{p}\,\mathrm{d}x,\;\mathrm{for\ all}\;v\in W_{0}^{1,p}(D), \label{D}%
\end{equation}
where $C_{N,p}=\lambda_{p}(B_{1})\left\vert B_{1}\right\vert ^{\frac{p}{N}},$
a positive constant that depends only on $p$ and $N.$

The following lemma is an adaptation of \cite[Theorem 4.1]{GECCM} which, in
its turn, is based on classical set level techniques (see \cite{Bandle,Ladyz}).

\begin{lemma}
\label{linfty}If $u\in W_{0}^{1,p}(\Omega)$ is a weak solution of
(\ref{singdiric}), then $u\in L^{\infty}(\Omega)$ and%
\begin{equation}
\left\Vert u\right\Vert _{\infty}\leq K_{N,p}\left\vert \Omega\right\vert
^{\frac{p}{N(p-q)}}\lambda^{\frac{1}{p-q}}, \label{x4b}%
\end{equation}
where $K_{N,p}$ is a positive constant depending only on $N$ and $p.$
\end{lemma}

\begin{proof}
For each $t>0,$ let
\[
E_{t}:=\left\{  x\in\Omega:u>t\right\}  \;\mathrm{and}\;\left(  u-t\right)
_{+}:=\max\left\{  u-t,0\right\}  \in W_{0}^{1,p}\left(  \Omega\right)  .
\]

Let us suppose $\left\vert E_{t}\right\vert >0.$ Since
\[
\int_{\Omega}\left\vert \nabla u\right\vert ^{p-2}\nabla u\cdot\nabla\left(
u-t\right)  _{+}\mathrm{d}x=\int_{E_{t}}\left\vert \nabla u\right\vert
^{p}\mathrm{d}x
\]
and%
\[
\int_{\Omega}u^{q-1}\left(  u-t\right)  _{+}\mathrm{d}x=\int_{E_{t}}%
u^{q-1}\left(  u-t\right)  \mathrm{d}x\leq t^{q-1}\int_{E_{t}}\left(
u-t\right)  \mathrm{d}x,
\]
(note that $q-1<0$) we obtain from (\ref{weak+}) that%
\begin{equation}
\int_{E_{t}}\left\vert \nabla u\right\vert ^{p}\mathrm{d}x\leq\lambda
t^{q-1}\int_{E_{t}}\left(  u-t\right)  \mathrm{d}x. \label{x1}%
\end{equation}

Now, we estimate $\int_{E_{t}}\left\vert \nabla u\right\vert ^{p}\mathrm{d}x$
from below. For this, we apply H\"{o}lder inequality and the estimate
(\ref{D}) with $D=E_{t}$ to obtain%
\[
\left(  \int_{E_{t}}\left(  u-t\right)  \mathrm{d}x\right)  ^{p}\leq
|E_{t}|^{p-1}\int_{E_{t}}\left(  u-t\right)  ^{p}\mathrm{d}x\leq\frac
{|E_{t}|^{p-1}|E_{t}|^{\frac{p}{N}}}{C_{N,p}}\int_{E_{t}}|\nabla
u|^{p}\,\mathrm{d}x.
\]
Hence,
\[
C_{N,p}|E_{t}|^{-\frac{p}{N}+1-p}\left(  \int_{E_{t}}\left(  u-t\right)
\mathrm{d}x\right)  ^{p}\leq\int_{E_{t}}\left\vert \nabla u\right\vert
^{p}\mathrm{d}x
\]
and, by taking into account (\ref{x1}), we get%
\[
C_{N,p}|E_{t}|^{-\frac{p}{N}+1-p}\left(  \int_{E_{t}}\left(  u-t\right)
\mathrm{d}x\right)  ^{p}\leq\lambda t^{q-1}\int_{E_{t}}\left(  u-t\right)
\mathrm{d}x,
\]
which is equivalent to%
\[
\left(  \int_{E_{t}}\left(  u-t\right)  \mathrm{d}x\right)  ^{p-1}\leq
\frac{\lambda}{C_{N,p}}t^{q-1}\left\vert E_{t}\right\vert ^{\frac{p+N(p-1)}%
{N}}.
\]

This latter inequality can be rewritten as%
\begin{equation}
\left(  \int_{E_{t}}\left(  u-t\right)  \mathrm{d}x\right)  ^{\frac
{N(p-1)}{p+N(p-1)}}\leq\left(  \frac{\lambda}{C_{N,p}}t^{q-1}\right)
^{\frac{N}{p+N(p-1)}}\left\vert E_{t}\right\vert . \label{x2}%
\end{equation}

Let us define
\[
f(t):=\int_{E_{t}}\left(  u-t\right)  \mathrm{d}x=\int_{t}^{\infty}\left\vert
E_{s}\right\vert \mathrm{d}s,
\]
where the second equality follows from Cavalieri's principle.

Since $f^{\prime}\left(  t\right)  =-\left\vert E_{t}\right\vert $ the
inequality in (\ref{x2}) can be rewritten as
\begin{equation}
t^{\frac{(1-q)N}{p+N(p-1)}}\leq-\left(  \frac{\lambda}{C_{N,p}}\right)
^{\frac{N}{p+N(p-1)}}f\left(  t\right)  ^{-\frac{N(p-1)}{p+N(p-1)}}f^{\prime
}\left(  t\right)  . \label{x3}%
\end{equation}

Integration of (\ref{x3}) yields%
\begin{align*}
\frac{p+N(p-1)}{p+N(p-q)}t^{\frac{p+N(p-q)}{p+N(p-1)}}  &  \leq\left(
\frac{\lambda}{C_{N,p}}\right)  ^{\frac{N}{p+N(p-1)}}\frac{p+N(p-1)}{p}\left[
f(0)^{\frac{p}{p+N(p-1)}}-f(t)^{\frac{p}{p+N(p-1)}}\right] \\
&  \leq\left(  \frac{\lambda}{C_{N,p}}\right)  ^{\frac{N}{p+N(p-1)}}%
\frac{p+N(p-1)}{p}\left(  \left\Vert u\right\Vert _{1}\right)  ^{\frac
{p}{p+N(p-1)}}<\infty.
\end{align*}
We have concluded that if $\left\vert E_{t}\right\vert >0$ then $t\leq K,$
where $K$ is a positive constant that does not depend on $t.$ Of course, this
implies that $\left\Vert u\right\Vert _{\infty}<\infty.$

Hence, we have%
\begin{align*}
\frac{p+N(p-1)}{p+N(p-q)}t^{\frac{p+N(p-q)}{p+N(p-1)}}  &  \leq\left(
\frac{\lambda}{C_{N,p}}\right)  ^{\frac{N}{p+N(p-1)}}\frac{p+N(p-1)}{p}\left(
\left\Vert u\right\Vert _{1}\right)  ^{\frac{p}{p+N(p-1)}}\\
&  \leq\left(  \frac{\lambda}{C_{N,p}}\right)  ^{\frac{N}{p+N(p-1)}}%
\frac{p+N(p-1)}{p}\left(  \left\Vert u\right\Vert _{\infty}\left\vert
\Omega\right\vert \right)  ^{\frac{p}{p+N(p-1)}}%
\end{align*}
and then, after making $t\rightarrow\left\Vert u\right\Vert _{\infty},$ we
obtain
\[
\left\Vert u\right\Vert _{\infty}\leq\left(  \frac{\lambda}{C_{N,p}}\right)
^{\frac{1}{p-q}}\left(  \frac{p+N(p-q)}{p}\right)  ^{\frac{p+N(p-1)}{N(p-q)}%
}\left\vert \Omega\right\vert ^{\frac{p}{N(p-q)}}%
\]
which leads to (\ref{x4b}) with%
\[
K_{N,p}:=\sup_{0\leq q\leq1}C_{N,p}^{-\frac{1}{p-q}}\left(  \frac{p+N(p-q)}%
{p}\right)  ^{\frac{p+N(p-1)}{N(p-q)}}.
\]

\end{proof}

In the next lemma, $\phi_{p}\in W_{0}^{1,p}(\Omega)$ denotes the $p$-torsion
function of $\Omega,$ that is, the weak solution of the $p$-torsional creep
problem%
\[
\left\{
\begin{array}
[c]{ll}%
-\Delta_{p}v=1 & \mathrm{in\ }\Omega,\\
v=0 & \mathrm{on\ }\partial\Omega.
\end{array}
\right.
\]
It is well known that the function $\phi_{p}$ is positive in $\Omega$ and
belongs to $C^{1,\alpha}(\overline{\Omega}),$ for some $\alpha\in(0,1).$

\begin{lemma}
\label{lowb}If $u\in W_{0}^{1,p}(\Omega)$ is a weak solution of
(\ref{singdiric}), then
\begin{equation}
0<\left(  K_{N,p}\left\vert \Omega\right\vert ^{\frac{p}{N(p-q)}}\right)
^{\frac{q-1}{p-1}}\lambda^{\frac{1}{p-q}}\phi_{p}(x)\leq u(x),\quad
\mathrm{for\ almost\ every}\;x\in\Omega. \label{a1}%
\end{equation}

\end{lemma}

\begin{proof}
Let $0\leq\varphi\in W_{0}^{1,p}(\Omega)$ and $c:=\left(  \lambda\left\Vert
u\right\Vert _{\infty}^{q-1}\right)  ^{\frac{1}{p-1}}.$ Since
\begin{align*}
\int_{\Omega}\left\vert \nabla u\right\vert ^{p-2}\nabla u\cdot\nabla
\varphi\mathrm{d}x  &  =\lambda\int_{\Omega}u^{q-1}\varphi\mathrm{d}x\\
&  \geq\lambda\left\Vert u\right\Vert _{\infty}^{q-1}\int_{\Omega}%
\varphi\mathrm{d}x\\
&  =\int_{\Omega}c^{p-1}\varphi\mathrm{d}x=\int_{\Omega}\left\vert
\nabla(c\phi_{p})\right\vert ^{p-2}\nabla(c\phi_{p})\cdot\nabla\varphi
\mathrm{d}x,
\end{align*}
the weak comparison principle guarantees that
\[
c\phi_{p}(x)\leq u(x),\quad\mathrm{for\ almost\ every}\;x\in\Omega.
\]
This leads to (\ref{a1}) since (\ref{x4b}) implies that
\[
c\geq\lambda^{\frac{1}{p-1}}(K_{N,p}\left\vert \Omega\right\vert ^{\frac
{p}{N(p-q)}}\lambda^{\frac{1}{p-q}})^{\frac{q-1}{p-1}}=\left(  K_{N,p}%
\left\vert \Omega\right\vert ^{\frac{p}{N(p-q)}}\right)  ^{\frac{q-1}{p-1}%
}\lambda^{\frac{1}{p-q}},\quad\mathrm{for\ almost\ every}\;x\in\Omega.
\]

\end{proof}

\begin{remark}
\label{regu}It follows from Lemma \ref{linfty} and Lemma \ref{lowb} that if
$u\in W_{0}^{1,p}(\Omega)$ is a weak solution of (\ref{singdiric}) then
\[
0<A\lambda^{\frac{p-1}{p-q}}\leq\lambda u(x)^{q-1}\leq B\lambda^{\frac
{p-1}{p-q}}\phi_{p}(x)^{q-1},\quad\mathrm{for\ almost\ every}\;x\in\Omega,
\]
where $A$ and $B$ are positive constants that depend only on $N,p$ and
$\left\vert \Omega\right\vert .$ This fact implies that if $\Omega^{\prime}$
is a subdomain of $\Omega$ such that $\overline{\Omega^{\prime}}\subset
\Omega,$ then $\lambda u^{q-1}$ is bounded in $\overline{\Omega^{\prime}}.$
\end{remark}

\section{The main results\label{S2}}

One can check, as a simple application of the H\"{o}lder inequality, that for
each $\ v\in L^{1}(\Omega)$ the function $q\in(0,1]\mapsto\left(  \frac
{1}{\left\vert \Omega\right\vert }\int_{\Omega}\left\vert v\right\vert
^{q}\mathrm{d}x\right)  ^{\frac{1}{q}}$ is increasing. This fact has two
immediate consequences: it implies that%
\[
0\leq\lim_{q\rightarrow0^{+}}\left(  \frac{1}{\left\vert \Omega\right\vert
}\int_{\Omega}\left\vert v\right\vert ^{q}\mathrm{d}x\right)  ^{\frac{1}{q}%
}=\inf_{0<s\leq1}\left(  \frac{1}{\left\vert \Omega\right\vert }\int_{\Omega
}\left\vert v\right\vert ^{s}\mathrm{d}x\right)  ^{\frac{1}{s}}\leq
\frac{\left\Vert v\right\Vert _{1}}{\left\vert \Omega\right\vert },\quad v\in
L^{1}(\Omega)
\]
and also that the function $q\in(0,1]\mapsto\lambda_{q}(\Omega)\left\vert
\Omega\right\vert ^{\frac{p}{q}}$ is decreasing, so that we can define%
\[
\mu(\Omega):=\lim_{q\rightarrow0^{+}}\lambda_{q}(\Omega)\left\vert
\Omega\right\vert ^{\frac{p}{q}}=\sup_{0<s\leq1}\lambda_{s}(\Omega)\left\vert
\Omega\right\vert ^{\frac{p}{s}}.
\]
Of course,
\begin{equation}
0<\lambda_{1}(\Omega)\left\vert \Omega\right\vert ^{p}\leq\mu(\Omega
)\leq\infty. \label{lowbmu}%
\end{equation}

Our first goal in this section is to show that $\mu(\Omega)<\infty.$

\begin{lemma}
One has%
\[
\lim_{q\rightarrow0^{+}}\left(  \int_{0}^{1}(1-t^{\frac{1}{q}})^{N}%
\mathrm{d}t\right)  ^{\frac{1}{q}}=e^{-1-\frac{1}{2}-\frac{1}{3}-\cdots
-\frac{1}{N}},\quad N\geq2.
\]

\end{lemma}

\begin{proof}
We have%
\[
\lim_{q\rightarrow0^{+}}\left(  \int_{0}^{1}(1-t^{\frac{1}{q}})^{N}%
\mathrm{d}t\right)  ^{\frac{1}{q}}=\lim_{s\rightarrow\infty}\left(  \int
_{0}^{1}(1-t^{s})^{N}\mathrm{d}t\right)  ^{s}=e^{L}.
\]
where%
\begin{align*}
L  &  :=\lim_{s\rightarrow\infty}\frac{\ln\left(  \int_{0}^{1}(1-t^{s}%
)^{N}\mathrm{d}t\right)  }{s^{-1}}\\
&  =\lim_{s\rightarrow\infty}\frac{s^{2}N\int_{0}^{1}(1-t^{s})^{N-1}t^{s}\ln
t\mathrm{d}t}{\int_{0}^{1}(1-t^{s})^{N}\mathrm{d}t}=\lim_{s\rightarrow\infty
}s^{2}N\int_{0}^{1}(1-t^{s})^{N-1}t^{s}\ln t\mathrm{d}t.
\end{align*}

After making the change of variable $\tau=t^{s}$ in the latter integral, we
obtain%
\[
L=N\lim_{s\rightarrow\infty}\int_{0}^{1}(1-\tau)^{N-1}\tau^{\frac{1}{s}}%
\ln\tau\mathrm{d}\tau=N\int_{0}^{1}(1-\tau)^{N-1}\ln\tau\mathrm{d}\tau.
\]

In order to finish the proof, it is enough to verify that%
\begin{equation}
N\int_{0}^{1}(1-\tau)^{N-1}\ln\tau\mathrm{d}\tau=-1-\frac{1}{2}-\frac{1}%
{3}-\cdots-\frac{1}{N},\quad N\geq2. \label{IN}%
\end{equation}
For this, let $I(N):=N\int_{0}^{1}(1-\tau)^{N-1}\ln\tau\mathrm{d}\tau.$ After
some simple calculations one can show that
\begin{equation}
I(N+1)=I(N)-\frac{1}{N+1},\quad N\geq2. \label{INa}%
\end{equation}

It is easy to check that $I(2)=-1-\dfrac{1}{2}.$ Hence, by using the recursive
formula (\ref{INa}), we arrive at (\ref{IN}).
\end{proof}

\begin{lemma}
\label{P3}Suppose that $\Omega$ is a bounded domain, star-shaped with respect
to $x_{0}\in\mathbb{R}^{N}.$ There exists $\rho\in C(\overline{\Omega})$ such
that: $0<\rho\leq1$ in $\Omega,$ $\rho(x_{0})=1,$ $\rho=0$ on $\partial\Omega$
and%
\[
\lim_{q\rightarrow0^{+}}\left(  \frac{1}{\left\vert \Omega\right\vert }%
\int_{\Omega}\left\vert \rho\right\vert ^{q}\mathrm{d}x\right)  ^{\frac{1}{q}%
}=e^{-1-\frac{1}{2}-\frac{1}{3}-\cdots-\frac{1}{N}}.
\]

In particular, any function $v\in W_{0}^{1,p}(\Omega)$ such that $v\geq
\rho\quad\mathrm{in\ }\Omega$ satisfies
\[
\lim_{q\rightarrow0^{+}}\left(  \frac{1}{\left\vert \Omega\right\vert }%
\int_{\Omega}\left\vert v\right\vert ^{q}\mathrm{d}x\right)  ^{\frac{1}{q}%
}\geq e^{-1-\frac{1}{2}-\frac{1}{3}-\cdots-\frac{1}{N}}>0.
\]

\end{lemma}

\begin{proof}
We will assume in this proof, without loss of generality, that $x_{0}=0.$

For each $0\not =x\in\overline{\Omega},$ let $r(x)$ be the unique positive
number such that
\[
r(x)x\in\partial\Omega.
\]
Of course, $r(x)\geq1$ and $r(x)\rightarrow\infty$ as $x\rightarrow0.$
Moreover, if $x\in\Omega$ and $\alpha>0$ is such that $\alpha x\in\Omega,$
then $r(\alpha x)\alpha x=r(x)x,$ so that%
\[
r(\alpha x)=\frac{r(x)}{\alpha}.
\]

Let us define $\rho:\overline{\Omega}\mapsto\lbrack0,1]$ by%
\[
\rho(x):=\left\{
\begin{array}
[c]{ll}%
1-\dfrac{1}{r(x)} & \mathrm{if}\quad x\in\overline{\Omega}\quad,\quad
x\not =0.\\
1 & \mathrm{if}\quad x=0.
\end{array}
\right.
\]
The graph of $\rho$ in $\mathbb{R}^{N}\times\mathbb{R}$ is the cone of base
$\Omega,$ height $1$ and vertex at the point $(0,1)\in\mathbb{R}^{N}%
\times\mathbb{R}.$

For each $t\in\lbrack0,1)$ the change of variable $x=(1-t)y$ yields
\begin{equation}
\left\vert \left\{  \rho(x)>t\right\}  \right\vert =\int_{\left\{
\rho(x)>t\right\}  }\mathrm{d}x=\int_{\left\{  \rho(y)>0\right\}  }%
(1-t)^{N}\mathrm{d}y=(1-t)^{N}\left\vert \Omega\right\vert . \label{cone1}%
\end{equation}
Indeed, by taking $\alpha=(1-t)$ one has%
\[
\rho(\alpha y)=1-\frac{1}{r(\alpha y)}=1-\frac{\alpha}{r(y)}=1-\alpha
+\alpha\left(  1-\frac{1}{r(y)}\right)  =1-\alpha+\alpha\rho(y).
\]
It follows that
\[
t<\rho((1-t)y)=t+(1-t)\rho(y)\Longleftrightarrow0<\rho(y).
\]

Thus, (\ref{cone1}) and Cavalieri's principle yield
\[
\int_{\Omega}\rho^{q}\mathrm{d}x=\int_{0}^{1}\left\vert \left\{  \rho
(x)^{q}>t\right\}  \right\vert \mathrm{d}t=\int_{0}^{1}\left\vert \left\{
\rho(x)>t^{\frac{1}{q}}\right\}  \right\vert \mathrm{d}t=\int_{0}%
^{1}(1-t^{\frac{1}{q}})^{N}\left\vert \Omega\right\vert \mathrm{d}t,
\]
so that%
\[
\lim_{q\rightarrow0^{+}}\left(  \frac{1}{\left\vert \Omega\right\vert }%
\int_{\Omega}\rho^{q}\mathrm{d}x\right)  ^{\frac{1}{q}}=\lim_{q\rightarrow
0^{+}}\left(  \int_{0}^{1}(1-t^{\frac{1}{q}})^{N}\mathrm{d}t\right)
^{\frac{1}{q}}=e^{-1-\frac{1}{2}-\frac{1}{3}-\cdots-\frac{1}{N}}.
\]

\end{proof}

\begin{remark}
If $\Omega=B_{R}$ is the ball centered at the origin with radius $R,$ then
$\rho(x)=1-\dfrac{\left\vert x\right\vert }{R}$ and%
\[
\frac{1}{\left\vert B_{R}\right\vert }\int_{B_{R}}\rho^{q}\mathrm{d}x=\int
_{0}^{1}(1-t^{\frac{1}{q}})^{N}\mathrm{d}t.
\]

\end{remark}

In the proof of the following theorem we will write $\Omega$ as a finite union
of star-shaped subdomains. This decomposition is quite general in the sense
that it is valid for bounded domains with low regularity as, for instance,
those with Lipschitz boundary (see \cite[Lemma II.1.3]{Galdi}).

\begin{theorem}
\label{main1}There exists $v\in W_{0}^{1,p}(\Omega)$ such that $v>0$ in
$\Omega$ and
\[
\lim_{q\rightarrow0^{+}}\left(  \frac{1}{\left\vert \Omega\right\vert }%
\int_{\Omega}\left\vert v\right\vert ^{q}\mathrm{d}x\right)  ^{\frac{1}{q}%
}=1.
\]

Moreover,%
\begin{equation}
\mu(\Omega)\leq\left\Vert \nabla v\right\Vert _{p}^{p}. \label{mufin}%
\end{equation}

\end{theorem}

\begin{proof}
Let $\Omega_{1},\Omega_{2},\ldots,\Omega_{m}$ be star-shaped subdomains of
$\Omega$ such that $\Omega=\bigcup\limits_{j=1}^{m}\Omega_{j}$ (not
necessarily disjoint).

According Lemma \ref{P3}, for each $j\in\Lambda:=\left\{  1,2,\ldots
,m\right\}  $ we can take $v_{j}\in W_{0}^{1,p}(\Omega_{j})$ such that
$v_{j}>0$ in $\Omega_{j}$ and
\[
e^{-1-\frac{1}{2}-\frac{1}{3}-\cdots-\frac{1}{N}}\leq\lim_{q\rightarrow0^{+}%
}\left(  \frac{1}{\left\vert \Omega_{j}\right\vert }\int_{\Omega_{j}%
}\left\vert v_{j}\right\vert ^{q}\mathrm{d}x\right)  ^{\frac{1}{q}}%
=\inf_{0<s\leq1}\left(  \frac{1}{\left\vert \Omega_{j}\right\vert }%
\int_{\Omega_{j}}\left\vert v_{j}\right\vert ^{s}\mathrm{d}x\right)
^{\frac{1}{s}}.
\]
Thus,%
\[
e^{-1-\frac{1}{2}-\frac{1}{3}-\cdots-\frac{1}{N}}\leq\left(  \frac
{1}{\left\vert \Omega_{j}\right\vert }\int_{\Omega_{j}}\left\vert
v_{j}\right\vert ^{q}\mathrm{d}x\right)  ^{\frac{1}{q}},\quad j\in
\Lambda,\quad0<q\leq1.
\]

By extending $v_{j}$ to zero outside $\Omega_{j}$ we can consider that $v_{j}$
belongs to $W_{0}^{1,p}(\Omega).$ Thus,
\[
V:=\sum_{j=1}^{m}v_{j}\in W_{0}^{1,p}(\Omega).
\]

Now, let $q_{n}\rightarrow0^{+}$ and, for each $n\in\mathbb{N},$ let $j_{n}%
\in\Lambda$ be such that
\[
\frac{1}{\left\vert \Omega_{j_{n}}\right\vert }\int_{\Omega_{j_{n}}}\left\vert
v_{j_{n}}\right\vert ^{q_{n}}\mathrm{d}x=\min\left\{  \frac{1}{\left\vert
\Omega_{j}\right\vert }\int_{\Omega_{j}}\left\vert v_{j}\right\vert ^{q_{n}%
}\mathrm{d}x:j\in\Lambda\right\}  .
\]
Then, for each fixed $n\in\mathbb{N}$ we have%
\begin{align*}
\int_{\Omega}\left\vert V\right\vert ^{q_{n}}\mathrm{d}x  &  =\sum_{j=1}%
^{m}\left(  \frac{1}{\left\vert \Omega_{j}\right\vert }\int_{\Omega_{j}%
}\left\vert v_{j}\right\vert ^{q_{n}}\mathrm{d}x\right)  \left\vert \Omega
_{j}\right\vert \\
&  \geq\frac{1}{\left\vert \Omega_{j_{n}}\right\vert }\int_{\Omega_{j_{n}}%
}\left\vert v_{j_{n}}\right\vert ^{q_{n}}\mathrm{d}x\sum_{j=1}^{m}\left\vert
\Omega_{j}\right\vert \\
&  \geq\frac{\left\vert \Omega\right\vert }{\left\vert \Omega_{j_{n}%
}\right\vert }\int_{\Omega_{j_{n}}}\left\vert v_{j_{n}}\right\vert ^{q_{n}%
}\mathrm{d}x\geq\left\vert \Omega\right\vert (e^{-1-\frac{1}{2}-\frac{1}%
{3}-\cdots-\frac{1}{N}})^{q_{n}},
\end{align*}
from which we conclude that
\[
\left(  \frac{1}{\left\vert \Omega\right\vert }\int_{\Omega}\left\vert
V\right\vert ^{q_{n}}\mathrm{d}x\right)  ^{\frac{1}{q_{n}}}\geq e^{-1-\frac
{1}{2}-\frac{1}{3}-\cdots-\frac{1}{N}}>0.
\]
It follows that%
\[
\theta:=\lim_{q\rightarrow0^{+}}\left(  \frac{1}{\left\vert \Omega\right\vert
}\int_{\Omega}\left\vert V\right\vert ^{q}\mathrm{d}x\right)  ^{\frac{1}{q}%
}\geq e^{-1-\frac{1}{2}-\frac{1}{3}-\cdots-\frac{1}{N}}>0.
\]

Therefore, the function $v:=\theta^{-1}V$ belongs to $W_{0}^{1,p}(\Omega),$ is
positive in $\Omega$ and satisfies%
\[
\lim_{q\rightarrow0^{+}}\left(  \frac{1}{\left\vert \Omega\right\vert }%
\int_{\Omega}\left\vert v\right\vert ^{q}\mathrm{d}x\right)  ^{\frac{1}{q}%
}=1.
\]

Now, (\ref{mufin}) follows immediately since%
\[
\mu(\Omega)=\lim_{q\rightarrow0^{+}}\lambda_{q}(\Omega)\left\vert
\Omega\right\vert ^{\frac{p}{q}}\leq\lim_{q\rightarrow0^{+}}\frac{\left\Vert
\nabla v\right\Vert _{p}^{p}}{\left(  \int_{\Omega}\left\vert v\right\vert
^{q}\mathrm{d}x\right)  ^{\frac{p}{q}}}\left\vert \Omega\right\vert ^{\frac
{p}{q}}=\left\Vert \nabla v\right\Vert _{p}^{p}.
\]

\end{proof}

It might be interesting to know an explicit lower bound for an abstract
minimum such as $\mu(\Omega).$ Thus, by combining (\ref{lowbmu}) with
(\ref{aux1}) we have
\begin{equation}
\lambda_{1}(\Omega^{\ast})\left\vert \Omega\right\vert ^{p}\leq\lambda
_{1}(\Omega)\left\vert \Omega\right\vert ^{p}\leq\mu(\Omega), \label{lowbmu1}%
\end{equation}
where $\Omega^{\ast}$ denotes the ball centered at the origin with radius
$R=(\left\vert \Omega\right\vert /\omega_{N})^{\frac{1}{N}},$ so that
$\left\vert \Omega^{\ast}\right\vert =\left\vert \Omega\right\vert .$ It is a
known fact (see \cite{GEJMAA}) that $\lambda_{1}(D)=\left\Vert \phi
_{p,D}\right\Vert _{1}^{1-p},$ where $D$ is a bounded domain and $\phi_{p,D}$
denotes its $p$-torsion function. Since the $p$-torsion function of a ball
$B_{R}$ of radius $R$ is explicitly given by
\[
\phi_{p,B_{R}}(x)=\frac{p-1}{p}N^{-\frac{1}{p-1}}(R^{\frac{p}{p-1}}-\left\vert
x\right\vert ^{\frac{p}{p-1}}),\quad0\leq\left\vert x\right\vert \leq R,
\]
we can compute $\lambda_{1}(\Omega^{\ast})$ explicitly and so obtain, from
(\ref{lowbmu1}), the following estimate%
\begin{equation}
N\left(  N+\frac{p}{p-1}\right)  ^{p-1}(\omega_{N})^{\frac{p}{N}}\left\vert
\Omega\right\vert ^{1-\frac{p}{N}}\leq\mu(\Omega). \label{lowbmu2}%
\end{equation}


For the sake of clarity, we will make use of the following scaling property in the
next proof:
\[
\lambda_{q}(\Omega)\left\vert \Omega\right\vert ^{\frac{p}{q}}=\left\vert
\Omega\right\vert ^{1-\frac{N}{p}}\lambda_{q}(\Omega_{1}) %
\]
where $\Omega_{1}:=\left\{  \left\vert \Omega\right\vert ^{-\frac{1}{N}}%
x:x\in\Omega\right\}  $ is such that $\left\vert \Omega_{1}\right\vert =1.$
Thus,
\begin{equation}
\mu(\Omega)=\left\vert \Omega\right\vert ^{1-\frac{N}{p}} \mu(\Omega_{1}).
\label{scaling}%
\end{equation}

Let us define,
\[
\mathcal{M}(\Omega):=\left\{  v\in W_{0}^{1,p}(\Omega):\lim_{q\rightarrow
0^{+}}\left(  \frac{1}{\left\vert \Omega\right\vert }\int_{\Omega}\left\vert
v\right\vert ^{q}\mathrm{d}x\right)  ^{\frac{1}{q}}=1\right\}  .
\]
It is easy to check that $\mathcal{M}(\Omega)$ has infinitely many elements by
combining Lemma \ref{P3} with the construction in the proof of Theorem
\ref{main1}.

As pointed out in the Introduction, for each $q\in(0,1)$ there exist
$\alpha_{q}\in(0,1)$ and $u_{q}\in W_{0}^{1,p}(\Omega)\cap C^{1,\alpha_{q}%
}(\overline{\Omega})$ such that
\begin{equation}
u_{q}>0\quad\mathrm{in}\quad\Omega,\quad\lambda_{q}(\Omega)=\left\Vert \nabla
u_{q}\right\Vert _{p}^{p},\quad\int_{\Omega}\left\vert u_{q}\right\vert
^{q}\mathrm{d}x=1\label{ann1}%
\end{equation}
and%
\begin{equation}
\int_{\Omega}\left\vert \nabla u_{q}\right\vert ^{p-2}\nabla u_{q}\cdot
\nabla\varphi\mathrm{d}x=\lambda_{q}(\Omega)\int_{\Omega}u_{q}^{q-1}%
\varphi\mathrm{d}x\quad\mathrm{for}\,\mathrm{all}\;\varphi\in W_{0}%
^{1,p}(\Omega).\label{ann2}%
\end{equation}
The existence of $u_{q}$ satisfying (\ref{ann1}) and (\ref{ann2}) is proved in
\cite{Anello}, whereas the H\"{o}lder regularity of $u_{q}$ follows directly
from \cite[Theorem 2.2 (i)]{GST1}. Let us observe that the proof of
(\ref{ann2}) made in \cite{Anello} is restricted to the functions $\varphi\in
W_{0}^{1,p}(\Omega)$ such that $\operatorname*{supp}\varphi\subset\Omega.$
However, this restriction can be dropped by using arguments of \cite{GST,
Mohammed} based on Fatou's lemma combined with the density of $C_{c}^{\infty
}(\Omega)$ in $W_{0}^{1,p}(\Omega).$ We will make use of these arguments in
the next proof.

\begin{theorem}
\label{theoexist}For each $q\in(0,1),$ let $u_{q}\in W_{0}^{1,p}(\Omega)\cap
C^{1,\alpha_{q}}(\overline{\Omega})$ satisfying (\ref{ann1}) and (\ref{ann2}).
There exists $u\in\mathcal{M}(\Omega)\cap C^{0,\alpha}(\overline{\Omega}),$
for some $\alpha\in(0,1),$ such that:

\begin{enumerate}
\item[(a)] $u=\lim\limits_{q\rightarrow0^{+}}(\left\vert \Omega\right\vert
^{\frac{1}{q}}u_{q})$ $\mathrm{in}$ $W_{0}^{1,p}(\Omega);$

\item[(b)] $\mu(\Omega)=\left\Vert \nabla u\right\Vert _{p}^{p}=\min\left\{
\left\Vert \nabla v\right\Vert _{p}^{p}:v\in\mathcal{M}(\Omega)\right\}  ;$

\item[(c)] $-\Delta_{p}u=\frac{\mu(\Omega)}{\left\vert \Omega\right\vert
}u^{-1},\quad\mathrm{in}\quad\Omega;$

\item[(d)] $0<A\mu(\Omega)^{\frac{1}{p}}\phi_{p}(x)\leq u(x)\leq B\mu
(\Omega)^{\frac{1}{p}},$ for almost every $x\in\Omega,$ where $A$ and $B$ are
positive constants depending only on $N,$ $p$ and $\left\vert \Omega
\right\vert .$
\end{enumerate}
\end{theorem}

\begin{proof}
Taking (\ref{scaling}) into account, we assume in this proof, without loss of
generality, that $\left\vert \Omega\right\vert =1.$ Thus,
\begin{equation}
\lim_{q\rightarrow0^{+}}\lambda_{q}(\Omega)=\lim_{q\rightarrow0^{+}}\left\Vert
\nabla u_{q}\right\Vert _{p}^{p}=\mu(\Omega)\in(0,\infty). \label{l01}%
\end{equation}

Since
\begin{equation}
\lambda_{q}(\Omega)\left(  \int_{\Omega}\left\vert v\right\vert ^{q}%
\mathrm{d}x\right)  ^{\frac{p}{q}}\leq\left\Vert \nabla v\right\Vert _{p}%
^{p}\quad\mathrm{for}\,\mathrm{all}\;v\in W_{0}^{1,p}(\Omega), \label{c}%
\end{equation}
we have
\begin{equation}
\mu(\Omega)\leq\left\Vert \nabla v\right\Vert _{p}^{p}\quad\mathrm{for}%
\,\mathrm{all}\;v\in\mathcal{M}(\Omega). \label{g}%
\end{equation}

It follows from (\ref{l01}) that there exist $q_{n}\rightarrow0^{+}$ and $u\in
W_{0}^{1,p}(\Omega)$ such that $u\geq0$ in $\Omega,$ $u_{q_{n}}\rightharpoonup
u$ (weakly) in $W_{0}^{1,p}(\Omega)$ and $u_{q_{n}}\rightarrow u$ pointwise
almost everywhere in $\Omega.$ Hence,%
\begin{equation}
\left\Vert \nabla u\right\Vert _{p}\leq\liminf_{n\rightarrow\infty}\left\Vert
\nabla u_{q_{n}}\right\Vert _{p}=\lim_{n\rightarrow\infty}\left\Vert \nabla
u_{q_{n}}\right\Vert _{p}=\lim_{n\rightarrow\infty}\lambda_{q_{n}}%
(\Omega)^{\frac{1}{p}}=\mu(\Omega)^{\frac{1}{p}}. \label{d}%
\end{equation}

We note from (\ref{c}), with $v=u,$ that
\[
\mu(\Omega)\lim_{q\rightarrow0^{+}}\left(  \int_{\Omega}\left\vert
u\right\vert ^{q}\mathrm{d}x\right)  ^{\frac{p}{q}}\leq\left\Vert \nabla
u\right\Vert _{p}^{p}.
\]
Combining this estimate with (\ref{d}) we obtain
\begin{equation}
\lim_{q\rightarrow0^{+}}\left(  \int_{\Omega}\left\vert u\right\vert
^{q}\mathrm{d}x\right)  ^{\frac{1}{q}}\leq1. \label{a}%
\end{equation}

On the other hand, for each $s\in(0,1)$ and every $n$ large enough (such that
$q_{n}<s$), we have%
\[
1=\left(  \int_{\Omega}\left\vert u_{q_{n}}\right\vert ^{q_{n}}\mathrm{d}%
x\right)  ^{\frac{1}{q_{n}}}\leq\left(  \int_{\Omega}\left\vert u_{q_{n}%
}\right\vert ^{s}\mathrm{d}x\right)  ^{\frac{1}{s}}.
\]
Hence,%
\[
1\leq\lim_{n\rightarrow\infty}\left(  \int_{\Omega}\left\vert u_{q_{n}%
}\right\vert ^{s}\mathrm{d}x\right)  ^{\frac{1}{s}}=\left(  \int_{\Omega
}\left\vert u\right\vert ^{s}\mathrm{d}x\right)  ^{\frac{1}{s}},
\]
where we have used Dominated Convergence Theorem, since%
\[
0\leq u_{q_{n}}\leq K_{N,p}\lambda_{q_{n}}(\Omega)^{\frac{1}{p-q_{n}}}\leq
K_{N,p}\mu(\Omega)^{\frac{1}{p-q_{n}}}%
\]
according Lemma \ref{linfty}. Thus, we conclude that
\begin{equation}
1\leq\lim_{s\rightarrow0^{+}}\left(  \int_{\Omega}\left\vert u\right\vert
^{s}\mathrm{d}x\right)  ^{\frac{1}{s}}. \label{b}%
\end{equation}

Gathering (\ref{a}) and (\ref{b}) we obtain%
\[
\lim_{q\rightarrow0^{+}}\left(  \int_{\Omega}\left\vert u\right\vert
^{q}\mathrm{d}x\right)  ^{\frac{1}{q}}=1.
\]
It follows that $u\in\mathcal{M}(\Omega)$ and thus, by combining (\ref{g}) and
(\ref{d}) we conclude that%
\begin{equation}
\mu(\Omega)^{\frac{1}{p}}=\left\Vert \nabla u\right\Vert _{p}=\lim
_{n\rightarrow\infty}\left\Vert \nabla u_{q_{n}}\right\Vert _{p}, \label{f}%
\end{equation}
which ends the proof of the claim \textrm{(b)}.

Taking into account the weak convergence $u_{q_{n}}\rightharpoonup u,$ the
second equality in (\ref{f}) implies that $u_{q_{n}}\rightarrow u$ (strongly)
in $W_{0}^{1,p}(\Omega).$ In view of (\ref{ann2}) we have%
\[
\int_{\Omega}\left\vert \nabla u_{q_{n}}\right\vert ^{p-2}\nabla u_{q_{n}%
}\cdot\nabla\varphi\mathrm{d}x=\lambda_{q_{n}}(\Omega)\int_{\Omega}u_{q_{n}%
}^{q_{n}-1}\varphi\mathrm{d}x,\quad\mathrm{for}\,\mathrm{all}\;\varphi\in
W_{0}^{1,p}(\Omega).
\]

Strong convergence $u_{q_{n}}\rightarrow u$ in $W_{0}^{1,p}(\Omega)$
guarantees that%
\begin{equation}
\lim_{n\rightarrow\infty}\int_{\Omega}\left\vert \nabla u_{q_{n}}\right\vert
^{p-2}\nabla u_{q_{n}}\cdot\nabla\varphi\mathrm{d}x=\int_{\Omega}\left\vert
\nabla u\right\vert ^{p-2}\nabla u\cdot\nabla\varphi\mathrm{d}x,\quad
\mathrm{for}\,\mathrm{all}\;\varphi\in W_{0}^{1,p}(\Omega) \label{h3}%
\end{equation}
and (\ref{l01}) guarantees that%
\begin{equation}
\lim_{n\rightarrow\infty}\lambda_{q_{n}}(\Omega)\int_{\Omega}u_{q_{n}}%
^{q_{n}-1}\varphi\mathrm{d}x=\mu(\Omega)\lim_{n\rightarrow\infty}\int_{\Omega
}u_{q_{n}}^{q_{n}-1}\varphi\mathrm{d}x,\quad\mathrm{for}\,\mathrm{all}%
\;\varphi\in W_{0}^{1,p}(\Omega). \label{h2}%
\end{equation}

Let us first assume that \textrm{supp\ }$\varphi\subset\Omega.$ Then,
Dominated Convergence Theorem yields
\begin{equation}
\lim_{n\rightarrow\infty}\int_{\Omega}u_{q_{n}}^{q_{n}-1}\varphi
\mathrm{d}x=\int_{\Omega}u^{-1}\varphi\mathrm{d}x, \label{h1}%
\end{equation}
since Lemma \ref{linfty} and Lemma \ref{lowb} imply that $0<c_{1}\leq
u_{q_{n}}^{q_{n}-1}\varphi\leq c_{2}$ in \textrm{supp\ }$\varphi,$ where the
constants $c_{1}$ and $c_{2}$ are uniform with respect to $n.$ Hence, by
gathering (\ref{h3}), (\ref{h2}) and (\ref{h1}) we have%
\begin{equation}
\int_{\Omega}\left\vert \nabla u\right\vert ^{p-2}\nabla u\cdot\nabla
\varphi\mathrm{d}x=\mu(\Omega)\int_{\Omega}\varphi u^{-1}\mathrm{d}%
x,\quad\mathrm{for}\,\mathrm{all}\;\varphi\in C_{c}^{\infty}(\Omega).
\label{h4}%
\end{equation}

Thus, in order to prove \textrm{(c) } we need to show that (\ref{h4}) holds,
in fact, for any $\varphi\in W_{0}^{1,p}(\Omega),$ which reduces to prove that
(\ref{h1}) holds for any $\varphi\in W_{0}^{1,p}(\Omega).$ We prove this by
following arguments of \cite{GST, Mohammed}. So, let $w\in W_{0}^{1,p}%
(\Omega)$ be arbitrary and take a sequence $\left\{  \xi_{n}\right\}  \subset
C_{c}^{\infty}(\Omega)$ of nonnegative functions such that $\xi_{n}%
\rightarrow\left\vert w\right\vert ,$ strongly in $W_{0}^{1,p}(\Omega)$ and
pointwise almost everywhere in $\Omega.$ Hence, by applying: Fatou's lemma,
(\ref{h4}) and H\"{o}lder's inequality, we obtain%
\begin{align*}
\left\vert \int_{\Omega}wu^{-1}\mathrm{d}x\right\vert  &  \leq\int_{\Omega
}\left\vert w\right\vert u^{-1}\mathrm{d}x\\
&  \leq\liminf_{n\rightarrow\infty}\int_{\Omega}\xi_{n}u^{-1}\mathrm{d}x\\
&  =\lim_{n\rightarrow\infty}\int_{\Omega}\left\vert \nabla u\right\vert
^{p-2}\nabla u\cdot\nabla\xi_{n}\mathrm{d}x\leq\left\Vert \nabla u\right\Vert
_{p}^{p-1}\lim_{n\rightarrow\infty}\left\Vert \nabla\xi_{n}\right\Vert
_{p}=\left\Vert \nabla u\right\Vert _{p}^{p-1}\left\Vert \nabla w\right\Vert
_{p}.
\end{align*}

Now, let $\varphi$ be an arbitrary function in $W_{0}^{1,p}(\Omega)$ and take
$\left\{  \varphi_{n}\right\}  \subset C_{c}^{\infty}(\Omega)$ such that
$\varphi_{n}\rightarrow\varphi$ strongly in $W_{0}^{1,p}(\Omega).$ Then, by
using $\varphi_{n}-\varphi$ in the place of $w,$ we obtain%
\[
\lim_{n\rightarrow\infty}\left\vert \int_{\Omega}(\varphi_{n}-\varphi
)u^{-1}\mathrm{d}x\right\vert \leq\left\Vert \nabla u\right\Vert _{p}%
^{p-1}\lim_{n\rightarrow\infty}\left\Vert \nabla(\varphi_{n}-\varphi
)\right\Vert _{p}=0,
\]
which yields
\begin{equation}
\lim_{n\rightarrow\infty}\int_{\Omega}\varphi_{n}u^{-1}\mathrm{d}%
x=\int_{\Omega}\varphi u^{-1}\mathrm{d}x. \label{h5}%
\end{equation}

Since $\varphi_{n}\in C_{c}^{\infty}(\Omega)$ we obtain from (\ref{h4}) that
\begin{equation}
\lim_{n\rightarrow\infty}\int_{\Omega}\varphi_{n}u^{-1}\mathrm{d}x=\mu
(\Omega)^{-1}\lim_{n\rightarrow\infty}\int_{\Omega}\left\vert \nabla
u\right\vert ^{p-2}\nabla u\cdot\nabla\varphi_{n}\mathrm{d}x=\mu(\Omega
)^{-1}\int_{\Omega}\left\vert \nabla u\right\vert ^{p-2}\nabla u\cdot
\nabla\varphi\mathrm{d}x. \label{h6}%
\end{equation}

Therefore, by combining (\ref{h5}) with (\ref{h6}) we conclude that (\ref{h1})
holds true for any $\varphi\in W_{0}^{1,p}(\Omega),$ which proves the claim
\textrm{(c)}.

Claim \textrm{(d) }now follows after combining Lemma \ref{linfty} with Lemma
\ref{lowb}. Theorem 2.2 (ii) of \cite{GST1} implies that $u\in C^{0,\alpha
}(\overline{\Omega})$ for some $\alpha\in(0,1).$

Claim \textrm{(a)} follows from the uniqueness of the weak solutions of
\begin{equation}
\left\{
\begin{array}
[c]{ll}%
-\Delta_{p}w=\frac{\mu(\Omega)}{\left\vert \Omega\right\vert }w^{-1} &
\mathrm{in\ }\Omega,\\
w>0 & \mathrm{in\ }\Omega,\\
w=0 & \mathrm{on\ }\partial\Omega,
\end{array}
\right.  \label{weakl0}%
\end{equation}
combined with the fact that $u_{q_{n}}\rightarrow u$ strongly in $W_{0}%
^{1,p}(\Omega).$ Indeed, these facts together imply that $u$ is the unique
limit function of the family $\left\{  u_{q}\right\}  ,$ as $q\rightarrow
0^{+}.$
\end{proof}

Our next goal is to prove that the solution $u$ of (\ref{weakl0}) satisfies
\[
\int_{\Omega}\log u\mathrm{d}x=0.
\]

\begin{proposition}
\label{into}Let $v\in L^{1}(\Omega).$ Then $\log\left\vert v\right\vert $ is
Lebesgue measurable in $\Omega$ and%
\[
-\infty\leq\lim_{q\rightarrow0^{+}}\int_{\Omega}\left\vert v\right\vert
^{q}\log\left\vert v\right\vert \mathrm{d}x=\int_{\Omega}\log\left\vert
v\right\vert \mathrm{d}x\leq2e^{-1}\left\Vert v\right\Vert _{1}.
\]

\end{proposition}

\begin{proof}
For every $x\in\Omega$ such that $\left\vert v(x)\right\vert \leq1$ the
function $q\in(0,1]\mapsto-\left\vert v(x)\right\vert ^{q}\log\left\vert
v(x)\right\vert \in\lbrack0,\infty]$ is decreasing and
\[
\lim_{q\rightarrow0^{+}}\left[  -\left\vert v(x)\right\vert ^{q}\log\left\vert
v(x)\right\vert \right]  =-\log\left\vert v(x)\right\vert .
\]
Therefore, it follows directly from Lebesgue's Monotone Convergence Theorem
that $\log\left\vert v\right\vert $ is Lebesgue measurable in the set
$\left\{  x\in\Omega:\left\vert v(x)\right\vert \leq1\right\}  $ and%
\[
-\infty\leq\lim_{q\rightarrow0^{+}}\int_{\left\{  \left\vert v\right\vert
\leq1\right\}  }\left\vert v\right\vert ^{q}\log\left\vert v\right\vert
\mathrm{d}x=\int_{\left\{  \left\vert v\right\vert \leq1\right\}  }%
\log\left\vert v\right\vert \mathrm{d}x\leq0.
\]

Now, for every $x\in\Omega$ such that $\left\vert v(x)\right\vert \geq1$ the
function $q\in(0,1]\mapsto\left\vert v(x)\right\vert ^{q}\log\left\vert
v(x)\right\vert \in\lbrack0,\infty]$ is increasing and
\[
\lim_{q\rightarrow0^{+}}\left\vert v(x)\right\vert ^{q}\log\left\vert
v(x)\right\vert =\log\left\vert v(x)\right\vert .
\]
Moreover, for every $q\in(0,\frac{1}{2})$ one has
\[
0\leq\left\vert v(x)\right\vert ^{q}\log\left\vert v(x)\right\vert
\leq\left\vert v(x)\right\vert ^{\frac{1}{2}}\log\left\vert v(x)\right\vert
\leq2e^{-1}\left\vert v(x)\right\vert ^{\frac{1}{2}}\left\vert v(x)\right\vert
^{\frac{1}{2}}=2e^{-1}\left\vert v(x)\right\vert \in L^{1}(\Omega),
\]
since $\max\limits_{t\geq1}t^{-\frac{1}{2}}\log t=2e^{-1}.$ Therefore,
Lebesgue's Dominated Convergence Theorem implies that $\log\left\vert
v\right\vert $ is integrable in the set $\left\{  x\in\Omega:\left\vert
v(x)\right\vert \geq1\right\}  $ and that%
\[
0\leq\lim_{q\rightarrow0^{+}}\int_{\left\{  \left\vert v\right\vert
\geq1\right\}  }\left\vert v\right\vert ^{q}\log\left\vert v\right\vert
\mathrm{d}x=\int_{\left\{  \left\vert v\right\vert \geq1\right\}  }%
\log\left\vert v\right\vert \mathrm{d}x\leq2e^{-1}\left\Vert v\right\Vert
_{1}.
\]

Thus, we have that $\log\left\vert v\right\vert $ is Lebesgue measurable in
$\Omega$ and
\begin{align*}
\lim_{q\rightarrow0^{+}}\int_{\Omega}\left\vert v\right\vert ^{q}%
\log\left\vert v\right\vert \mathrm{d}x  &  =\lim_{q\rightarrow0^{+}}%
\int_{\left\{  \left\vert v\right\vert \leq1\right\}  }\left\vert v\right\vert
^{q}\log\left\vert v\right\vert \mathrm{d}x+\lim_{q\rightarrow0^{+}}%
\int_{\left\{  \left\vert v\right\vert >1\right\}  }\left\vert v\right\vert
^{q}\log\left\vert v\right\vert \mathrm{d}x\\
&  =\int_{\left\{  \left\vert v\right\vert \leq1\right\}  }\log\left\vert
v\right\vert \mathrm{d}x+\int_{\left\{  \left\vert v\right\vert >1\right\}
}\log\left\vert v\right\vert \mathrm{d}x\\
&  =\int_{\Omega}\log\left\vert v\right\vert \mathrm{d}x\leq2e^{-1}\left\Vert
v\right\Vert _{1}.
\end{align*}

\end{proof}

\begin{proposition}
\label{intlog1}For $v\in L^{1}(\Omega)$ define
\[
\theta_{v}:=\lim_{q\rightarrow0^{+}}\left(  \frac{1}{\left\vert \Omega
\right\vert }\int_{\Omega}\left\vert v\right\vert ^{q}\mathrm{d}x\right)
^{\frac{1}{q}}\in\lbrack0,\infty)\;\mathrm{and}\;\beta_{v}:=\frac
{1}{\left\vert \Omega\right\vert }\int_{\Omega}\log\left\vert v\right\vert
\mathrm{d}x\in\lbrack-\infty,\infty).
\]
Then
\begin{equation}
\beta_{v}=\log\theta_{v}. \label{logl}%
\end{equation}

\end{proposition}

\begin{proof}
It follows from L'H\^{o}pital's rule and Proposition \ref{into} that%
\[
\lim_{q\rightarrow0^{+}}\left(  \frac{1}{\left\vert \Omega\right\vert }%
\int_{\Omega}\left\vert v\right\vert ^{q}\mathrm{d}x\right)  ^{\frac{1}{q}%
}=\exp\left(  \frac{1}{\left\vert \Omega\right\vert }\lim_{q\rightarrow0^{+}%
}\int_{\Omega}\left\vert v\right\vert ^{q}\log\left\vert v\right\vert
\mathrm{d}x\right)  =\exp\left(  \frac{1}{\left\vert \Omega\right\vert }%
\int_{\Omega}\log\left\vert v\right\vert \mathrm{d}x\right)  .
\]
Hence, (\ref{logl}) follows.
\end{proof}

The following corollary is an immediate consequence of the previous proposition.

\begin{corollary}
\label{log0}Let $v\in W_{0}^{1,p}(\Omega).$ Then $v\in\mathcal{M}(\Omega)$ if,
and only if,
\[
\int_{\Omega}\log\left\vert v\right\vert \mathrm{d}x=0.
\]
In particular, $\int_{\Omega}\log\left\vert u\right\vert \mathrm{d}x=0,$ where
$u$ is the solution of (\ref{weakl0}).
\end{corollary}

\subsection{Minimizing the energy functional}

In this subsection, $u$ denotes the solution of (\ref{weakl0}). As we have
shown, $u$ minimizes the functional $v\mapsto\left\Vert \nabla v\right\Vert
_{p}^{p}$ on $\mathcal{M}(\Omega).$ Let us show that $u$ is the unique, up to
sign, with this property.

\begin{lemma}
\label{unicM0}Let $v\in\mathcal{M}(\Omega)$ such that $\mu(\Omega)=\left\Vert
\nabla v\right\Vert _{p}^{p}.$ Then, $v$ does not change sign in $\Omega.$
\end{lemma}

\begin{proof}
Let $\Omega_{+}:=\left\{  x\in\Omega:v(x)>0\right\}  ,$ $\Omega_{-}=\left\{
x\in\Omega:v(x)<0\right\}  ,$ $a_{+}:=\frac{\left\vert \Omega_{+}\right\vert
}{\left\vert \Omega\right\vert }$ and $a_{-}:=\frac{\left\vert \Omega
_{-}\right\vert }{\left\vert \Omega\right\vert }.$ For $0<q<1,$ we have%
\begin{align*}
\left(  \frac{1}{\left\vert \Omega\right\vert }\int_{\Omega}\left\vert
v\right\vert ^{q}\mathrm{d}x\right)  ^{\frac{1}{q}}  &  =\left(  \frac{a_{+}%
}{\left\vert \Omega_{+}\right\vert }\int_{\Omega_{+}}\left\vert v_{+}%
\right\vert ^{q}\mathrm{d}x+\frac{a_{-}}{\left\vert \Omega_{-}\right\vert
}\int_{\Omega_{-}}\left\vert v_{-}\right\vert ^{q}\mathrm{d}x\right)
^{\frac{1}{q}}\\
&  \leq a_{+}\left(  \frac{1}{\left\vert \Omega_{+}\right\vert }\int
_{\Omega_{+}}\left\vert v_{+}\right\vert ^{q}\mathrm{d}x\right)  ^{\frac{1}%
{q}}+a_{-}\left(  \frac{1}{\left\vert \Omega_{-}\right\vert }\int_{\Omega_{-}%
}\left\vert v_{-}\right\vert ^{q}\mathrm{d}x\right)  ^{\frac{1}{q}}.
\end{align*}

It follows that%
\begin{align*}
1 &  =\lim_{q\rightarrow0^{+}}\left(  \frac{1}{\left\vert \Omega\right\vert
}\int_{\Omega}\left\vert v\right\vert ^{q}\mathrm{d}x\right)  ^{\frac{1}{q}}\\
&  \leq(a_{+})\lim_{q\rightarrow0^{+}}\left(  \frac{1}{\left\vert \Omega
_{+}\right\vert }\int_{\Omega_{+}}\left\vert v_{+}\right\vert ^{q}%
\mathrm{d}x\right)  ^{\frac{1}{q}}+(a_{-})\lim_{q\rightarrow0^{+}}\left(
\frac{1}{\left\vert \Omega_{-}\right\vert }\int_{\Omega_{-}}\left\vert
v_{-}\right\vert ^{q}\mathrm{d}x\right)  ^{\frac{1}{q}}.
\end{align*}

Let us suppose, by contradiction, that both $\left\vert \Omega_{+}\right\vert
$ and $\left\vert \Omega_{-}\right\vert $ are positive. Without loss of
generality, we assume that
\[
\lim_{q\rightarrow0^{+}}\left(  \frac{1}{\left\vert \Omega_{-}\right\vert
}\int_{\Omega_{-}}\left\vert v_{-}\right\vert ^{q}\mathrm{d}x\right)
^{\frac{1}{q}}\leq\theta_{+}:=\lim_{q\rightarrow0^{+}}\left(  \frac
{1}{\left\vert \Omega_{+}\right\vert }\int_{\Omega_{+}}\left\vert
v_{+}\right\vert ^{q}\mathrm{d}x\right)  ^{\frac{1}{q}}.
\]
Then%
\[
1\leq(a_{+}+a_{-})\lim_{q\rightarrow0^{+}}\left(  \frac{1}{\left\vert
\Omega_{+}\right\vert }\int_{\Omega_{+}}\left\vert v_{+}\right\vert
^{q}\mathrm{d}x\right)  ^{\frac{1}{q}}=\theta_{+}.
\]

It follows that $\theta_{+}^{-1}v_{+}\in\mathcal{M}(\Omega)$ and%
\[
\mu(\Omega)\leq\left\Vert \nabla(\theta_{+}^{-1}v_{+})\right\Vert _{p}^{p}%
\leq\left\Vert \nabla(v_{+})\right\Vert _{p}^{p}\leq\left\Vert \nabla
v\right\Vert _{p}^{p}=\mu(\Omega).
\]
Thus,
\[
\mu(\Omega)=\left\Vert \nabla v_{+}\right\Vert _{p}^{p}=\left\Vert \nabla
v\right\Vert _{p}^{p}=\left\Vert \nabla v_{+}\right\Vert _{p}^{p}+\left\Vert
\nabla v_{-}\right\Vert _{p}^{p},
\]
implying that $\left\Vert \nabla v_{-}\right\Vert _{p}^{p}=0$ and, therefore,
that $v_{-}=0$ $\mathrm{a.e}$\textrm{.} in $\Omega.$ This implies that
$\left\vert \Omega_{-}\right\vert =0,$ which is a contradiction.
\end{proof}

\begin{theorem}
\label{unicM}Let $v\in\mathcal{M}(\Omega)$ such that $\mu(\Omega)=\left\Vert
\nabla v\right\Vert _{p}^{p}.$ Then, $v=\pm u$ $\mathrm{a.e}$\textrm{.} in
$\Omega.$
\end{theorem}

\begin{proof}
According Lemma \ref{unicM0}, we can assume, without loss of generality, that
$v\geq0$ \textrm{a.\thinspace e.} in $\Omega.$ Since $0<q<1,$ it is simple to
check that
\[
\left(  \int_{\Omega}\left(  \frac{u+v}{2}\right)  ^{q}\mathrm{d}x\right)
^{\frac{1}{q}}\geq\left(  \int_{\Omega}\left(  \frac{u}{2}\right)
^{q}\mathrm{d}x\right)  ^{\frac{1}{q}}+\left(  \int_{\Omega}\left(  \frac
{v}{2}\right)  ^{q}\mathrm{d}x\right)  ^{\frac{1}{q}}.
\]
Thus,%
\begin{equation}
h:=\lim_{q\rightarrow0^{+}}\left(  \frac{1}{\left\vert \Omega\right\vert }%
\int_{\Omega}\left(  \frac{u+v}{2}\right)  ^{q}\mathrm{d}x\right)  ^{\frac
{1}{q}}\geq\frac{1}{2}+\frac{1}{2}=1. \label{h+}%
\end{equation}
Of course, $\dfrac{u+v}{2h}\in\mathcal{M}(\Omega).$ Thus,
\begin{equation}
\mu(\Omega)\leq\left\Vert \nabla\left(  \dfrac{u+v}{2h}\right)  \right\Vert
_{p}^{p}\leq\frac{1}{(2h)^{p}}\left(  \left\Vert \nabla u\right\Vert
_{p}+\left\Vert \nabla v\right\Vert _{p}\right)  ^{p}=\frac{\left(
2\mu(\Omega)^{\frac{1}{p}}\right)  ^{p}}{(2h)^{p}}=\frac{\mu(\Omega)}{h^{p}},
\label{h-}%
\end{equation}
which implies that $h\leq1.$ Therefore, (\ref{h+}) and (\ref{h-}) imply that
$h=1,$ and this fact in (\ref{h-}) yields%
\[
\left\Vert \nabla(u+v)\right\Vert _{p}=\left\Vert \nabla u\right\Vert
_{p}+\left\Vert \nabla v\right\Vert _{p}.
\]
We conclude from this equality that $u=v$ \textrm{a.\thinspace e.} in
$\Omega.$
\end{proof}

The next corollary shows that $\mu(\Omega)^{-1}$ is the best constant $C$ in
the following log-Sobolev type inequality%
\begin{equation}
\exp\left(  \frac{1}{\left\vert \Omega\right\vert }\int_{\Omega}\log\left\vert
v\right\vert ^{p}\mathrm{d}x\right)  \leq C\left\Vert \nabla v\right\Vert
_{p}^{p},\quad v\in W_{0}^{1,p}(\Omega)\label{logsob}%
\end{equation}
and that when $C=\mu(\Omega)^{-1}$ this inequality becomes an equality if, and
only if, $v$ is a scalar multiple of $u.$

\begin{corollary}
One has%
\[
\mu(\Omega)=\min\left\{  \frac{\left\Vert \nabla v\right\Vert _{p}^{p}}%
{\exp\left(  \frac{1}{\left\vert \Omega\right\vert }\int_{\Omega}%
\log\left\vert v\right\vert ^{p}\mathrm{d}x\right)  }:v\in W_{0}^{1,p}%
(\Omega)\;\mathrm{and}\;\int_{\Omega}\log\left\vert v\right\vert
^{p}\mathrm{d}x>-\infty\right\}  .
\]
Moreover, the minimum is reached only by scalar multiples of $u.$
\end{corollary}

\begin{proof}
Let $v\in W_{0}^{1,p}(\Omega)$ be such that $\int_{\Omega}\log\left\vert
v\right\vert ^{p}\mathrm{d}x>-\infty.$ Since
\[
\lambda_{q}(\Omega)\left\vert \Omega\right\vert ^{\frac{p}{q}}\leq
\frac{\left\Vert \nabla v\right\Vert _{p}^{p}}{\left(  \frac{1}{\left\vert
\Omega\right\vert }\int_{\Omega}\left\vert v\right\vert ^{q}\mathrm{d}%
x\right)  ^{\frac{p}{q}}}%
\]
we obtain%
\[
\mu(\Omega)=\lim_{q\rightarrow0^{+}}\lambda_{q}(\Omega)\left\vert
\Omega\right\vert ^{\frac{p}{q}}\leq\frac{\left\Vert \nabla v\right\Vert
_{p}^{p}}{\lim\limits_{q\rightarrow0^{+}}\left(  \frac{1}{\left\vert
\Omega\right\vert }\int_{\Omega}\left\vert v\right\vert ^{q}\mathrm{d}%
x\right)  ^{\frac{p}{q}}}=\frac{\left\Vert \nabla v\right\Vert _{p}^{p}}%
{\exp\left(  \frac{1}{\left\vert \Omega\right\vert }\int_{\Omega}%
\log\left\vert v\right\vert ^{p}\mathrm{d}x\right)  }.
\]

Since $\int_{\Omega}\log\left\vert u\right\vert \mathrm{d}x=0,$ we have%
\[
\frac{\left\Vert \nabla u\right\Vert _{p}^{p}}{\exp\left(  \frac{1}{\left\vert
\Omega\right\vert }\int_{\Omega}\log\left\vert u\right\vert ^{p}%
\mathrm{d}x\right)  }=\left\Vert \nabla u\right\Vert _{p}^{p}=\mu(\Omega).
\]

Of course, if $v\in W_{0}^{1,p}(\Omega)$ is such that
\[
\frac{\left\Vert \nabla v\right\Vert _{p}^{p}}{\exp\left(  \frac{1}{\left\vert
\Omega\right\vert }\int_{\Omega}\log\left\vert v\right\vert ^{p}%
\mathrm{d}x\right)  }=\mu(\Omega)
\]
then $\theta_{v}^{-1}v\in\mathcal{M}(\Omega),$ where
\[
\theta_{v}=\lim\limits_{q\rightarrow0^{+}}\left(  \frac{1}{\left\vert
\Omega\right\vert }\int_{\Omega}\left\vert v\right\vert ^{q}\mathrm{d}%
x\right)  ^{\frac{1}{q}}=\exp\left(  \frac{1}{\left\vert \Omega\right\vert
}\int_{\Omega}\log\left\vert v\right\vert ^{p}\mathrm{d}x\right)  >0,
\]
and $\left\Vert \nabla(\theta_{v}^{-1}v)\right\Vert _{p}^{p}=\mu(\Omega).$
Hence, $\theta_{v}^{-1}v=\pm u,$ implying that $v$ is a scalar multiple of
$u.$
\end{proof}

\begin{remark}
Gathering (\ref{lowbmu2}) and (\ref{logsob}), with $C=\mu(\Omega)^{-1},$ we
obtain%
\[
\exp\left(  \frac{1}{\left\vert \Omega\right\vert }\int_{\Omega}\log\left\vert
v\right\vert ^{p}\mathrm{d}x\right)  \leq C_{N,p,\left\vert \Omega\right\vert
}\left\Vert \nabla v\right\Vert _{p}^{p},\quad v\in W_{0}^{1,p}(\Omega),
\]
where $C_{N,p,\left\vert \Omega\right\vert }:=N^{-1}\left(  N+\frac{p}%
{p-1}\right)  ^{1-p}(\omega_{N})^{-\frac{p}{N}}\left\vert \Omega\right\vert
^{\frac{p}{N}-1}.$
\end{remark}

Now, let us define $J:W_{0}^{1,p}(\Omega)\mapsto(-\infty,\infty],$ the formal
energy functional, by%
\[
J(v):=\left\{
\begin{array}
[c]{ll}%
\dfrac{1}{p}%
{\displaystyle\int_{\Omega}}
\left\vert \nabla v\right\vert ^{p}\mathrm{d}x-\dfrac{\mu(\Omega)}{\left\vert
\Omega\right\vert }%
{\displaystyle\int_{\Omega}}
\log\left\vert v\right\vert \mathrm{d}x, & \mathrm{if}\;%
{\displaystyle\int_{\Omega}}
\log\left\vert v\right\vert \mathrm{d}x\in(-\infty,\infty)\\
\infty, & \mathrm{if}\;%
{\displaystyle\int_{\Omega}}
\log\left\vert v\right\vert \mathrm{d}x=-\infty.
\end{array}
\right.
\]

We are going to show that $u$ is the unique minimizer of $J$ in $W_{0}%
^{1,p}(\Omega).$

\begin{lemma}
\label{intlog2}One has%
\begin{equation}
J(v)\geq\frac{\mu(\Omega)}{p}-\frac{\mu(\Omega)}{p}\log\left(  \frac
{\mu(\Omega)}{\left\Vert \nabla v\right\Vert _{p}^{p}}\right)  -\frac
{\mu(\Omega)}{\left\vert \Omega\right\vert }\int_{\Omega}\log\left\vert
v\right\vert \mathrm{d}x,\quad\mathrm{for}\,\mathrm{all}\;v\in W_{0}%
^{1,p}(\Omega). \label{Jmin0}%
\end{equation}

\end{lemma}

\begin{proof}
Let $v\in W_{0}^{1,p}(\Omega)$ and consider the function%
\[
g(t):=\frac{t^{p}}{p}\left\Vert \nabla v\right\Vert _{p}^{p}-\mu(\Omega)\log
t-\frac{\mu(\Omega)}{\left\vert \Omega\right\vert }\int_{\Omega}\log\left\vert
v\right\vert \mathrm{d}x,\quad t>0.
\]
It is easy to check that $J(tv)=g(\left\vert t\right\vert )$ and that
$\min_{t>0}g(t)$ is the right-hand side of (\ref{Jmin0}). The result then
follows, since $J(v)=g(1)\geq\min_{t>0}g(t).$
\end{proof}

\begin{theorem}
\label{intlog3}One has
\[
J(u)=\frac{\mu(\Omega)}{p}=\min\limits_{v\in W_{0}^{1,p}(\Omega)}J(v).
\]

\end{theorem}

\begin{proof}
Since
\[
\mu(\Omega)=%
{\displaystyle\int_{\Omega}}
\left\vert \nabla u\right\vert ^{p}\mathrm{d}x\;\mathrm{and}\;%
{\displaystyle\int_{\Omega}}
\log\left\vert u\right\vert \mathrm{d}x=0
\]
we have $J(u)=\dfrac{\mu(\Omega)}{p}.$ Thus, we need only to prove that%
\begin{equation}
J(v)\geq\frac{\mu(\Omega)}{p},\quad\mathrm{for}\,\mathrm{all}\;v\in
W_{0}^{1,p}(\Omega).\label{Jmin1}%
\end{equation}

Let $v\in W_{0}^{1,p}(\Omega)$ and $\theta_{v}$ as in Proposition
\ref{intlog1}. If $\theta_{v}=0,$ then $J(v)=\infty$ and (\ref{Jmin1}) holds
trivially. If $\theta_{v}>0$ then $\theta_{v}^{-1}v\in\mathcal{M}(\Omega)$
and
\[
\frac{1}{\left\vert \Omega\right\vert }\int_{\Omega}\log\left\vert \theta
_{v}^{-1}v\right\vert \mathrm{d}x=0\;\mathrm{and}\;\mu(\Omega)\leq\left\Vert
\nabla(\theta_{v}^{-1}v)\right\Vert _{p}^{p}=\theta_{v}^{-p}\left\Vert \nabla
v\right\Vert _{p}^{p},
\]
implying, respectively, that
\[
\frac{1}{\left\vert \Omega\right\vert }\int_{\Omega}\log\left\vert
v\right\vert \mathrm{d}x=\log\theta_{v}%
\]
and
\begin{align*}
J(v)  &  \geq\frac{\mu(\Omega)}{p}-\frac{\mu(\Omega)}{p}\log\left(  \frac
{\mu(\Omega)}{\left\Vert \nabla v\right\Vert _{p}^{p}}\right)  -\mu
(\Omega)\log(\theta_{v})\\
&  \geq\frac{\mu(\Omega)}{p}-\frac{\mu(\Omega)}{p}\log\left(  \theta_{v}%
^{-p}\right)  -\mu(\Omega)\log(\theta_{v})=\frac{\mu(\Omega)}{p}.
\end{align*}

\end{proof}

\begin{theorem}
If $w\in W_{0}^{1,p}(\Omega)$ minimizes $J,$ then $w=\pm\theta_{w}u$
$\mathrm{a.e.}$ in $\Omega,$ where
\[
\theta_{w}:=\lim_{q\rightarrow0^{+}}\left(  \frac{1}{\left\vert \Omega
\right\vert }\int_{\Omega}\left\vert w\right\vert ^{q}\mathrm{d}x\right)
^{\frac{1}{q}}.
\]

\end{theorem}

\begin{proof}
It follows from Lemma \ref{intlog2} that
\[
\dfrac{\mu(\Omega)}{p}=J(w)\geq\frac{\mu(\Omega)}{p}-\frac{\mu(\Omega)}{p}%
\log\left(  \frac{\mu(\Omega)}{\left\Vert \nabla w\right\Vert _{p}^{p}%
}\right)  -\frac{\mu(\Omega)}{\left\vert \Omega\right\vert }\int_{\Omega}%
\log\left\vert w\right\vert \mathrm{d}x,
\]
which implies that%
\begin{equation}
\frac{1}{p}\log\left(  \frac{\mu(\Omega)}{\left\Vert \nabla w\right\Vert
_{p}^{p}}\right)  +\frac{1}{\left\vert \Omega\right\vert }\int_{\Omega}%
\log\left\vert w\right\vert \mathrm{d}x\geq0. \label{Jmin2}%
\end{equation}

Of course, $\theta_{w}\in(0,\infty),$ so that $\theta_{w}^{-1}w\in
\mathcal{M}(\Omega).$ Thus, according Corollary \ref{log0},
\[
\frac{1}{\left\vert \Omega\right\vert }\int_{\Omega}\log\left\vert \theta
_{w}^{-1}w\right\vert \mathrm{d}x=0,
\]
so that%
\[
\log\theta_{w}=\frac{1}{\left\vert \Omega\right\vert }\int_{\Omega}%
\log\left\vert w\right\vert \mathrm{d}x.
\]
Hence, (\ref{Jmin2}) yields
\[
\frac{1}{p}\log\left(  \frac{\mu(\Omega)}{\left\Vert \nabla w\right\Vert
_{p}^{p}}\right)  +\log\theta_{w}=\log\left(  \frac{\mu(\Omega)}{\left\Vert
\nabla(\theta_{w}^{-1}w)\right\Vert _{p}^{p}}\right)  ^{\frac{1}{p}}\geq0,
\]
which is equivalent to $\mu(\Omega)\geq\left\Vert \nabla\theta_{w}%
^{-1}w\right\Vert _{p}^{p}.$

Since $\mu(\Omega)\leq\left\Vert \nabla\theta_{w}^{-1}w\right\Vert _{p}^{p}$
(recall that $\theta_{w}^{-1}w\in\mathcal{M}(\Omega)$), we conclude that
$\mu(\Omega)=\left\Vert \nabla(\theta_{w}^{-1}w)\right\Vert _{p}^{p}.$ So,
Theorem \ref{unicM} yields $\theta_{w}^{-1}w=\pm u.$
\end{proof}

We end this section by remarking that, for each $\lambda>0,$ a simple scaling
argument shows that the function
\[
u_{\lambda}:=\left(  \frac{\lambda\left\vert \Omega\right\vert }{\mu(\Omega
)}\right)  ^{\frac{1}{p}}u
\]
is the unique solution of the singular problem%
\[
\left\{
\begin{array}
[c]{ll}%
-\Delta_{p}v=\lambda v^{-1} & \mathrm{in\ }\Omega,\\
v>0 & \mathrm{in\ }\Omega,\\
v=0 & \mathrm{on\ }\partial\Omega.
\end{array}
\right.
\]
Moreover, since $\int_{\Omega}\log u\mathrm{d}x=0$ and $\mu(\Omega)=\left\Vert
\nabla u\right\Vert _{p}^{p},$ the equality in (\ref{logsob}) (with
$C=\mu(\Omega)^{-1}$) yields%
\[
\int_{\Omega}\log u_{\lambda}\mathrm{d}x=\frac{\left\vert \Omega\right\vert
}{p}\log\left(  \mu(\Omega)^{-1}\left(  \frac{\lambda\left\vert \Omega
\right\vert }{\mu(\Omega)}\right)  \left\Vert \nabla u\right\Vert _{p}%
^{p}\right)  =\frac{\left\vert \Omega\right\vert }{p}\log\left(  \frac
{\lambda\left\vert \Omega\right\vert }{\mu(\Omega)}\right)  .
\]

We also note that $u_{\lambda}$ and $-u_{\lambda}$ are the unique minimizers
of the functional
\[
J_{\lambda}(v):=\left\{
\begin{array}
[c]{ll}%
\dfrac{1}{p}%
{\displaystyle\int_{\Omega}}
\left\vert \nabla v\right\vert ^{p}\mathrm{d}x-\lambda%
{\displaystyle\int_{\Omega}}
\log\left\vert v\right\vert \mathrm{d}x, & \mathrm{if}\;%
{\displaystyle\int_{\Omega}}
\log\left\vert v\right\vert \mathrm{d}x\in(-\infty,\infty)\\
\infty, & \mathrm{if}\;%
{\displaystyle\int_{\Omega}}
\log\left\vert v\right\vert \mathrm{d}x=-\infty,
\end{array}
\right.
\]
being
\[
J_{\lambda}(\pm u_{\lambda})=\frac{\mu(\Omega)}{p}-\frac{\lambda\left\vert
\Omega\right\vert }{p}\log\left(  \frac{\lambda\left\vert \Omega\right\vert
}{\mu(\Omega)}\right)  .
\]

\subsection{Asymptotics for the pair $(\lambda_{q}(\Omega),\left\Vert
u_{q}\right\Vert _{\infty})$}

In this subsection we describe the asymptotic behavior of $\lambda_{q}%
(\Omega),$ as $q\rightarrow0^{+}.$ Of course, if $\left\vert \Omega\right\vert
=1$, then $\lim_{q\rightarrow0^{+}}\lambda_{q}(\Omega)=\mu(\Omega)$ and,
according items \textrm{(a)} and \textrm{(d)} of Theorem \ref{theoexist}%
\[
0<A\mu(\Omega)^{\frac{1}{p}}\left\Vert \phi_{p}\right\Vert _{\infty}\leq
\lim_{q\rightarrow0^{+}}\left\Vert u_{q}\right\Vert _{\infty}\leq B\mu
(\Omega)^{\frac{1}{p}},
\]
where $A$ and $B$ are positive constants depending only on $N$ and $p.$ If
$\left\vert \Omega\right\vert \not =1$ we have%
\[
\lim_{q\rightarrow0^{+}}\lambda_{q}(\Omega)=\lim_{q\rightarrow0^{+}}%
(\lambda_{q}(\Omega)\left\vert \Omega\right\vert ^{\frac{p}{q}})\left(
\lim_{q\rightarrow0^{+}}\left\vert \Omega\right\vert ^{-\frac{p}{q}}\right)
=\mu(\Omega)\left(  \lim_{q\rightarrow0^{+}}\left\vert \Omega\right\vert
^{-\frac{p}{q}}\right)  .
\]
Therefore, in this case, we readily obtain%
\begin{equation}
\lim_{q\rightarrow0^{+}}\lambda_{q}(\Omega)=\left\{
\begin{array}
[c]{ll}%
\infty & \mathrm{if}\quad\left\vert \Omega\right\vert <1\\
0 & \mathrm{if}\quad\left\vert \Omega\right\vert >1.
\end{array}
\right.  \label{limlamb}%
\end{equation}

Hence, by combining (\ref{limlamb}) with Lemma \ref{linfty} and Lemma
\ref{lowb} we conclude that
\[
\lim_{q\rightarrow0^{+}}\left\Vert u_{q}\right\Vert _{\infty}=\left\{
\begin{array}
[c]{ll}%
\infty & \mathrm{if}\quad\left\vert \Omega\right\vert <1\\
0 & \mathrm{if}\quad\left\vert \Omega\right\vert >1.
\end{array}
\right.
\]

We remark that (\ref{limlamb}) is also simple to prove without using Theorem
\ref{theoexist}. In fact, in the case $\left\vert \Omega\right\vert <1$ the
monotonicity of the function $q\mapsto\lambda_{q}(\Omega)\left\vert
\Omega\right\vert ^{\frac{p}{q}}$ implies that
\[
\lim_{q\rightarrow0^{+}}\lambda_{q}(\Omega)=\lim_{q\rightarrow0^{+}}%
\lambda_{q}(\Omega)\left\vert \Omega\right\vert ^{\frac{p}{q}}\left\vert
\Omega\right\vert ^{-\frac{p}{q}}\geq\lambda_{1}(\Omega)\left\vert
\Omega\right\vert ^{p}\lim_{q\rightarrow0^{+}}\left\vert \Omega\right\vert
^{-\frac{p}{q}}=\infty.
\]
This is the proof given in \cite{Anello}. As for the case $\left\vert
\Omega\right\vert >1,$ take $v\in W_{0}^{1,p}(\Omega)\cap C(\overline{\Omega
})$ such that $v>0$ in $\Omega$ and then define $\Omega_{\epsilon}:=\left\{
x\in\Omega:v(x)\geq\epsilon\right\}  ,$ where $\epsilon>0$ is such that
$\left\vert \Omega_{\epsilon}\right\vert >1.$Then,
\[
0<\lambda_{q}(\Omega)\leq\frac{\int_{\Omega}\left\vert \nabla v\right\vert
^{p}\mathrm{d}x}{\left(  \int_{\Omega}\left\vert v\right\vert ^{q}%
\mathrm{d}x\right)  ^{\frac{p}{q}}}\leq\frac{\int_{\Omega}\left\vert \nabla
v\right\vert ^{p}\mathrm{d}x}{\left(  \int_{\Omega_{\epsilon}}\epsilon
^{q}\mathrm{d}x\right)  ^{\frac{p}{q}}}=\frac{\int_{\Omega}\left\vert \nabla
v\right\vert ^{p}\mathrm{d}x}{\epsilon^{p}\left\vert \Omega_{\epsilon
}\right\vert ^{\frac{p}{q}}}.
\]
Thus, by making $q\rightarrow0^{+},$ we obtain $\lim_{q\rightarrow0^{+}%
}\lambda_{q}(\Omega)=0$, since $\lim_{q\rightarrow0^{+}}\left\vert
\Omega_{\epsilon}\right\vert ^{-\frac{p}{q}}=0.$

\section{Acknowledgments}

The first author thanks the support of Funda\c{c}\~{a}o de Amparo \`{a}
Pesquisa do Estado de Minas Gerais (Fapemig)/Brazil (CEX-PPM-00165) and
Conselho Nacional de Desenvolvimento Cient\'{\i}fico e Tecnol\'{o}gico
(CNPq)/Brazil (483970/2013-1 and 306590/2014-0).

\end{document}